\documentclass[a4paper,leqno,10pt]{amsart}

\usepackage{epsf,graphicx,epsfig}
\usepackage{amscd}
\usepackage{amsmath,latexsym,amssymb,amsthm}
\usepackage[nospace,noadjust]{cite}
\usepackage{textcomp}
\usepackage{setspace,cite}
\usepackage{lscape,fancyhdr,fancybox}
\usepackage{stmaryrd}
\usepackage[all,cmtip]{xy}
\usepackage{tikz}
\usepackage{cancel}
\usetikzlibrary{shapes,arrows,decorations.markings}
\usepackage{graphicx}

\usepackage{color}
\usepackage{url}
\usepackage{enumerate}
\usepackage[mathscr]{euscript}

  \theoremstyle{plain}

\swapnumbers
    \newtheorem{thm}{Theorem}[section]
    \newtheorem{prop}[thm]{Proposition}

    \newtheorem{subsec}[thm]{}
\theoremstyle{definition}
    \newtheorem{defn}[thm]{Definition}
        \newtheorem{remark}[thm]{Remark}
    \newtheorem{exam}[thm]{Example}

\theoremstyle{remark}

\title{}
\author{}
\date{}
\usepackage{amssymb}

\usepackage{hyperref}
\hypersetup{
	colorlinks,
	citecolor=blue,
	filecolor=black,
	linkcolor=blue,
	urlcolor=black
}

\begin{document}
\title{Rota-Baxter systems on Hom-associative algebras and covariant Hom-bialgebras\footnote{2020 {\em MSC.} 17B38, 17B61, 16T10}.}
\author{Apurba Das}
\address{Department of Mathematics and Statistics,
Indian Institute of Technology, Kanpur 208016, Uttar Pradesh, India.}
\email{apurbadas348@gmail.com}


\keywords{Hom-associative algebras, Rota-Baxter systems, Hom-Yang-Baxter pairs, Covariant Hom-bialgebras}

\begin{abstract}
Rota-Baxter systems were introduced by Brzezi\'{n}ski as a generalization of Rota-Baxter operators that are related to dendriform structures, associative Yang-Baxter pairs and covariant bialgebras. In this paper, we define Rota-Baxter systems on Hom-associative algebras and show how they induce Hom-dendriform structures and weak pseudotwistors on the underlying Hom-associative algebra. We introduce Hom-Yang-Baxter pairs and a notion of covariant Hom-bialgebra. Given a Hom-Yang-Baxter pair, we construct a twisted Rota-Baxter system and a (quasitriangular) covariant Hom-bialgebra. Finally, we consider perturbations of the coproduct in a covariant Hom-bialgebra.
\end{abstract}

\maketitle

\tableofcontents
 




\section{Introduction}

Rota-Baxter operators are algebraic abstraction of the integral operator. A Rota-Baxter operator on an algebra $A$ is a linear map $R: A \rightarrow A$ that satisfies
\begin{align*}
R(a) \cdot R(b) = R( R(a) \cdot b + a \cdot R(b)),~\text{ for } a, b \in A.
\end{align*}
They were first appeared in the fluctuation theory in probability by G. Baxter \cite{baxter} and further studied by G.-C. Rota \cite{rota} and P. Cartier \cite{cart} in relationship with combinatorics. An important application of Rota-Baxter operators were found in the Connes-Kreimer's algebraic approach of the renormalization of quantum field theory \cite{conn}. In \cite{aguiar-pre} M. Aguiar relates Rota-Baxter operators with dendriform algebras introduced by J.-L. Loday \cite{loday}. See \cite{fard-guo,das-rota,guo-book,guo-keigher} for some recent study on Rota-Baxter operators.

\medskip 

On the other hand, the notion of infinitesimal bialgebras was introduced by Joni and Rota \cite{joni-rota} to give an algebraic framework for the calculus of Newton divided differences and further studied by Aguiar \cite{aguiar}. Infinitesimal bialgebras are the associative analog of Lie bialgebras and related to Rota-Baxter operators. Motivated from the classical Yang-Baxter equation in Lie algebra context, in \cite{aguiar}, Aguiar introduced an associative analog of Yang-Baxter equation. A solution of the associative Yang-Baxter equation induces a Rota-Baxter operator and an infinitesimal bialgebra, called quasitriangular infinitesimal bialgebra.

\medskip

In \cite{brz} Brzezi\'{n}ski introduced Rota-Baxter systems as a generalization of Rota-Baxter operator that are also related to dendriform structures. Brzezi\'{n}ski also considered a notion of associative Yang-Baxter pair (generalization of associative Yang-Baxter solution), a notion of covariant bialgebra (generalization of infinitesimal bialgebra) and find various relations between them. These relations are generalization of Aguiar's results that can be mentioned as follows:
\begin{center}
$\text{infinitesimal bialgebra}  \longleftrightarrow \text{associative Yang-Baxter solution}  \longleftrightarrow  \text{Rota-Baxter operator}  \longleftrightarrow   \text{Dendriform algebra}$
\end{center}
\begin{center}
$\text{covariant bialgebra}  \longleftrightarrow \text{associative Yang-Baxter pair}  \longleftrightarrow  \text{Rota-Baxter system}  \longleftrightarrow   \text{Dendriform algebra}.$
\end{center}

\medskip

Recently, Hom-type algebras are widely studied in the literature by many authors. In these algebras, the identities defining the structures are twisted by linear maps \cite{makh-sil}. The notion of Hom-Lie algebras was introduced by Hartwig, Larsson and Silvestrov that appeared in examples of $q$-deformations of the Witt and Virasoro algebras \cite{hls}. Other types of algebras (e.g. associative, Leibniz, dendriform, Poisson etc.)
twisted by linear maps have also been studied extensively \cite{lie-theory,makh-hom-rota,das-colloq,das-cr,makh-sil}. Here we mainly put our attention to Hom-associative algebras. Various aspects of Hom-associative algebras were studied in \cite{lie-theory,das-cr,das-hom-inf}. In \cite{yau} Yau defined associative Hom-Yang-Baxter equation, infinitesimal Hom-bialgebras and generalize some results of Aguiar \cite{aguiar}. Rota-Baxter operators on Hom-associative algebras was defined by Makhlouf in \cite{makh-hom-rota} and relations with Hom-dendriform algebras and Hom-preLie algebras was obtained.

\medskip

In this paper, we define Rota-Baxter system on Hom-associative algebras generalizing Rota-Baxter Hom-associative algebras. We show that a Rota-Baxter system induces a Hom-dendriform structure and therefore a Hom-preLie algebra structure. We show that a Rota-Baxter system on a Hom-associative algebra induces a weak pseudotwistor on the underlying Hom-associative algebra in the sense of \cite{liu}.

\medskip

Next, we introduce an associative Hom-Yang-Baxter pair which are Hom-analog of Yang-Baxter pairs of Brzezi\'{n}ski \cite{brz}. Associative Hom-Yang-Baxter pairs generalize associative Hom-Yang-Baxter solutions of \cite{yau}. We show that a Hom-Yang-Baxter pair induces a twisted version of Rota-Baxter system on the underlying Hom-associative algebra. Hence, by previous results, a Hom-Yang-Baxter pair induces a Hom-dendriform algebra and a Hom-preLie algebra.

\medskip

Finally, we define a notion of covariant Hom-bialgebra. They are generalization of covariant bialgebras and infinitesimal Hom-bialgebras. Given a covariant bialgebra and a morphism, one can construct a covariant Hom-bialgebra, called `induced by composition'. We show that a Hom-Yang-Baxter pair induces a covariant Hom-bialgebra, called quasitriangular covariant Hom-bialgebra. We give some characterizations of covariant Hom-bialgebras.  Finally, we consider perturbations of the coproduct in a covariant Hom-bialgebra generalizing the perturbation theory of Drinfel'd for quasi-Hopf algebras \cite{perturb1,perturb2}.

\medskip

All vector spaces, linear maps and tensor products are over any field $\mathbb{K}$.

\section{Hom-associative algebras and infinitesimal Hom-bialgebras}

In this section, we recall Hom-associative algebras and infinitesimal Hom-bialgebras. Our main references are \cite{makh-sil,yau}.

A Hom-vector space is a pair $(A, \alpha)$ in which $A$ is a vector space and $\alpha : A \rightarrow A$ is a linear map. A morphism between Hom-vector spaces $(A, \alpha)$ and $(B, \beta)$ is given by a linear map $f: A \rightarrow B$ satisfying $\beta \circ f = f \circ \alpha$. Hom-vector spaces and morphisms between them forms a category.

\begin{defn}
A Hom-associative algebra consists of a Hom-vector space $(A, \alpha)$ together with a bilinear map $\mu : A \otimes A \rightarrow A,~ a \otimes b \mapsto a \cdot b$ satisfying
\begin{align*}
(a \cdot b) \cdot \alpha (c) = \alpha (a) \cdot (b \cdot c),~ \text{ for } a, b, c \in A.
\end{align*}
A Hom-associative algebra as above may be denoted by $(A, \alpha, \mu)$ or simply by $A$. 
\end{defn}

A Hom-associative algebra $(A, \alpha, \mu)$ is said to be multiplicative if $\alpha $ preserves the multiplication, i.e. $\alpha ( a \cdot  b) = \alpha (a) \cdot \alpha (b)$, for $a, b \in A$.
From now onward, by a Hom-associative algebra, we shall always mean a multiplicative Hom-associative algebra. A Hom-associative algebra $(A, \alpha, \mu)$ with $\alpha = \mathrm{id}$ is simply an associative algebra. Therefore, the class of associative algebras is a subclass of the class of Hom-associative algebras.

\begin{defn}
Let $(A, \mu)$ be an associative algebra and $\alpha :A \rightarrow A$ be an associative algebra morphism. Then $(A, \alpha, \mu_\alpha = \alpha \circ \mu)$ is a Hom-associative algebra, called `induced by composition'.
\end{defn}

The dual notion of Hom-associative algebras is given by Hom-associative coalgebras.

\begin{defn}
A Hom-coassociative coalgebra is a Hom-vector space $(C, \alpha)$ together with a linear map $\triangle : C \rightarrow C \otimes C$ satisfying
\begin{align}
(\triangle \otimes \alpha ) \circ \triangle = (\alpha \otimes \triangle) \circ \triangle.
\end{align}
A Hom-coassociative coalgebra as above may be denoted by $(C, \alpha, \triangle)$ or simply by $C$.
A Hom-coassociative coalgebra $(C, \alpha, \triangle)$ is called multiplicative if $(\alpha \otimes \alpha ) \circ \triangle = \triangle \circ \alpha$. 
\end{defn}

In the rest of the paper, by a Hom-coassociative coalgebra, we shall always mean a multiplicative Hom-coassociative coalgebra.

\begin{defn}
An infinitesimal Hom-bialgebra is a quadruple $(A, \alpha, \mu, \triangle)$ in which $(A, \alpha, \mu)$ is a Hom-associative algebra, $(A, \alpha, \triangle)$ is a Hom-coassociative coalgebra and satisfying the following compatibility
\begin{align}\label{inf-hom-bi-comp}
\triangle \circ \mu = (\mu \otimes \alpha ) \circ (\alpha \otimes \triangle) + (\alpha \otimes \mu) \circ (\triangle \otimes \alpha).
\end{align}
\end{defn}

Using Sweedler's notation $\triangle (a) = a_{(1)} \otimes a_{(2)}$ (note that we do not use the summation sign in the Sweedler's notation which is hidden in the picture), the compatibility (\ref{inf-hom-bi-comp}) can be expressed as
\begin{align*}
\triangle (ab) = \alpha (a) \cdot b_{(1)} \otimes \alpha (b_{(2)}) + \alpha (a_{(1)}) \otimes a_{(2)} \cdot \alpha (b),~ \text{ for } a, b \in A.
\end{align*} 

Infinitesimal Hom-bialgebras are Hom-analog of infinitesimal bialgebras introduced by Joni and Rota \cite{joni-rota}, further studied by Aguiar \cite{aguiar}. More precisely, an infinitesimal Hom-bialgebra $(A, \alpha, \mu, \triangle)$ with $\alpha = \mathrm{id}$ is an infinitesimal bialgebra and a morphism between infinitesimal bialgebras gives rise to an infinitesimal Hom-bialgebra induced by composition \cite{yau}.

\section{Rota-Baxter systems on Hom-associative algebras}
In this section, we introduce Rota-Baxter systems on Hom-associative algebras and prove that they induce Hom-dendriform algebras. A Rota-Baxter system on a Hom-associative algebra gives rise to a weak pseudotwistor on the underlying Hom-associative algebra in the sense of \cite{liu}. 

\begin{defn}\label{defn-rb-s}
Let $(A, \alpha, \mu)$ be a Hom-associative algebra. A Rota-Baxter system on $A$ consists of a pair $(R, S)$ of linear maps on $A$ satisfying $\alpha \circ R = R \circ \alpha$, $\alpha \circ S = S \circ \alpha$ and
\begin{align}
R(a) \cdot R(b) =~& R ( R(a) \cdot b + a \cdot S(b)), \label{rbs1}\\
S(a) \cdot S(b) =~& S (R(a) \cdot b + a \cdot S(b)).  \label{rbs2}
\end{align}
\end{defn}

If the Hom-associative algebra is given by an associative algebra (i.e. $\alpha = \mathrm{id}$), one gets Rota-Baxter system on associative algebra introduced by Brzezi\'{n}ski \cite{brz}. 

In \cite{makh-hom-rota} the author introduced Rota-Baxter operators on Hom-associative algebras. A linear map $R : A \rightarrow A$ is a Rota-Baxter operator of weight $\lambda \in \mathbb{K}$ on a Hom-associative algebra $(A, \alpha, \mu)$ if $\alpha \circ R = R\circ \alpha$ and satisfying
\begin{align*}
R(a) \cdot R(b) = R ( R(a) \cdot b + a \cdot R(b) + \lambda a \cdot b),~ \text{ for } a, b \in A.
\end{align*}
It follows that a Rota-Baxter operator $R$ of weight $0$ is a Rota-Baxter system $(R, R)$. More generally, if $R$ is a Rota-Baxter operator of weight $\lambda$, then $(R, S= R + \lambda \mathrm{id})$ is a Rota-Baxter system. First observe that (\ref{rbs1}) holds obviously for $R=R$ and $S= R + \lambda \mathrm{id}$. To prove (\ref{rbs2}), we see that
\begin{align*}
S (R(a) \cdot b + a \cdot S(b)) =~& (R + \lambda \mathrm{id}) ( R(a) \cdot b + a \cdot R(b) + \lambda a \cdot b) \\
=~& R ( R(a) \cdot b + a \cdot R(b) + \lambda a \cdot b) + \lambda R(a) \cdot b +  \lambda a \cdot R(b) + \lambda^2 a \cdot b \\
=~& R(a) \cdot R(b) + \lambda R(a) \cdot b +  \lambda a \cdot R(b) + \lambda^2 a \cdot b \\
=~& (R(a) + \lambda a) \cdot (R(b) + \lambda b) = S(a) \cdot S(b).
\end{align*}

\begin{defn} \cite{makh-hom-rota,das-colloq}
A Hom-dendriform algebra is a Hom-vector space $(D, \alpha)$ together with bilinear maps $\prec, \succ : D \otimes D \rightarrow D$ satisfying $\alpha ( a \prec b) =\alpha(a) \prec \alpha (b)$, $\alpha (a \succ b) = \alpha (a) \succ \alpha (b)$ and the following identities are hold
\begin{align*}
(a \prec b ) \prec \alpha (c) =~& \alpha (a) \prec ( b \prec c + b \succ c ),\\
(a \succ b ) \prec \alpha (c) =~& \alpha (a) \succ (b \prec c),\\
(a \prec b + a \succ b) \succ \alpha (c) =~&  \alpha (a) \succ (b \succ c).
\end{align*}
\end{defn}

A Hom-dendriform algebra as above may be denoted by $(D, \alpha, \prec, \succ)$. Hom-dendriform algebras Hom-analog of Loday's dendriform algebras \cite{loday}. 

If $(D, \alpha, \prec, \succ)$ is a Hom-dendriform algebra then it follows that $(D, \alpha, *)$ is a Hom-associative algebra, where $a * b = a \prec b + a \succ b$, for $a , b \in D$.

\begin{defn}\cite{makh-hom-rota}
A Hom-preLie algebra is a Hom-vector space $(A, \alpha)$ together with a bilinear map $\diamond : A \otimes A \rightarrow A$ that satisfies $\alpha ( a \diamond b ) = \alpha (a) \diamond \alpha (b)$ and the following identity
\begin{align*}
(a \diamond b) \diamond \alpha (c) - \alpha (a)  \diamond  (b   \diamond    c)  = (b   \diamond  a ) \diamond   \alpha (c) - \alpha (b)     \diamond (a   \diamond  c),~ \text{ for } a, b, c \in A. 
\end{align*}
\end{defn}

The following result can be found in \cite{makh-hom-rota}.

\begin{prop}\label{hom-dend-prelie}
Let $(D, \alpha, \prec, \succ)$ be a Hom-dendriform algebra. Then the triple $(D, \alpha, \triangle)$ is a Hom-preLie algebra, where $a \diamond b = a \succ b - a \prec b$, for $a, b \in D$.
\end{prop}

\begin{prop}\label{rbs-hom-dend}
Let $(A, \alpha, \mu)$ be a Hom-associative algebra and $(R, S)$ be a Rota-Baxter system. Then $(A, \alpha, \prec, \succ)$ is a Hom-dendriform algebra, where $a \prec b = a \cdot S(b)$ and $a \succ b = R(a) \cdot b$, for $a, b \in A$.
\end{prop}

\begin{proof}
First observe that 
\begin{align*}
&\alpha (a \prec b ) = \alpha ( a \cdot S(b)) = \alpha (a) \cdot S(\alpha (b)) = \alpha (a) \prec \alpha (b),\\
&\alpha ( a \succ b ) = \alpha ( R(a) \cdot b ) = R (\alpha (a) ) \cdot \alpha (b) = \alpha (a) \succ \alpha (b).
\end{align*}
We have
\begin{align*}
(a \prec b) \prec \alpha (c) = (a \cdot S(b)) \cdot \alpha(S(c)) =~& \alpha (a) \cdot ( S(b) \cdot S(c)) \\
=~& \alpha (a) \cdot S ( b \cdot S(c) + R(b) \cdot c) \\
=~& \alpha (a) \prec ( b \prec c + b \succ c ).
\end{align*}
Similarly,
\begin{align*}
(a \succ b ) \prec \alpha (c) = ( R(a) \cdot b ) \cdot \alpha (S(c)) = R(\alpha (a) ) \cdot ( b \cdot S(c)) = \alpha (a) \succ ( b \prec c),
\end{align*}
and
\begin{align*}
(a \prec b + a \succ b ) \succ \alpha (c) = R ( a \cdot S(b) + R(a) \cdot b ) \cdot \alpha (c) =~& (R(a) \cdot R(b)) \cdot \alpha (c) \\
=~& R(\alpha (a)) \cdot ( R(b) \cdot c) \\=~& \alpha (a) \succ (b \succ c).
\end{align*}
Hence the proof.
\end{proof}

\begin{remark}
\begin{itemize}
\item[(i)] A Rota-Baxter system $(R, S)$ on a Hom-associative algebra $(A, \alpha, \mu)$ induces a new Hom-associative structure on $(A, \alpha)$ with multiplication 
\begin{align}\label{new-hom-ass}.
a * b = a \cdot S(b) + R(a) \cdot b, ~ \text{ for } a, b \in A.
\end{align}
\item[(ii)] It follows from the above Proposition and Proposition \ref{hom-dend-prelie} that a Rota-Baxter system $(R,S)$ on a Hom-associative algebra $(A, \alpha, \mu)$ induces a Hom-preLie algebra $(A, \alpha, \diamond)$, where $a \diamond b = R(a) \cdot b - a \cdot S(b)$, for $a, b \in A$.
\end{itemize}
\end{remark}

\begin{defn}
\cite{liu} Let $(A, \alpha, \mu)$ be a Hom-associative algebra. A linear map $T : A^{\otimes 2} \rightarrow A^{ \otimes 2}$ is called a weak pseudotwistor if $\alpha^{\otimes 2} \circ T = T \circ \alpha^{\otimes 2}$ and there exists a linear map $\tau : A^{\otimes 3} \rightarrow A^{\otimes 3}$ (called a weak companion of $T$) with $\alpha^{\otimes 3} \circ \tau = \tau \circ \alpha^{\otimes 3}$ that makes the following diagram commutative
\begin{align}\label{weak-pent}
\xymatrix{
 & A \otimes A \otimes A \ar[r]^{\alpha \otimes \mu}    &    A \otimes A  \ar[dd]^T   &  A \otimes A \otimes A \ar[l]_{\mu \otimes \alpha } &  \\
 A \otimes A \otimes A  \ar[ru]^{\mathrm{id} \otimes T} \ar[rd]_\tau & & & &  A \otimes A \ar[lu]_{T \otimes \mathrm{id}} \ar[ld]^{\tau} \\
  & A \otimes A \otimes A  \ar[r]_{\alpha \otimes \mu}  &    A \otimes A    &  A \otimes A \otimes A. \ar[l]^{\mu \otimes \alpha} & 
}
\end{align}
\end{defn}

In this case, it can be seen that $(A, \alpha, \mu \circ T)$ is a new Hom-associative algebra \cite{liu}.

\begin{prop}
Let $(R, S)$ be a Rota-Baxter system on a Hom-associative algebra $(A, \alpha, \mu)$. Then the map $T : A^{\otimes 2} \rightarrow A^{\otimes 2}$ defined by
\begin{align*}
T( a \otimes b) = R(a) \otimes b + a \otimes S(b)
\end{align*}
is a weak pseudotwistor.
\end{prop}

\begin{proof}
Let us define a map $\tau : A^{\otimes 3} \rightarrow A^{\otimes 3}$ by
\begin{align*}
\tau ( a \otimes b \otimes c ) = R(a) \otimes R(b) \otimes c + R(a) \otimes b \otimes S(c) + a \otimes S(b) \otimes S(c).
\end{align*}
We show that $T$ is a weak pseudotwistor with a weak companion $\tau$. First observe that $\alpha^{\otimes 2} \circ T = T \circ \alpha^{\otimes 2}$ and $\alpha^{\otimes 3} \circ \tau = \tau \circ \alpha^{\otimes 3}$. Next we have
\begin{align}\label{weak-1}
((\mu \otimes \alpha) \circ \tau) (a \otimes b \otimes c) =~& R(a) \cdot R(b) \otimes \alpha (c) + R(a) \cdot b \otimes S(\alpha (c)) + a \cdot S(b) \otimes S(\alpha (c)) \nonumber \\
=~& R (R(a) \cdot b) \otimes \alpha (c) + R (a \cdot S(b))\otimes \alpha (c) + R(a) \cdot b \otimes S(\alpha (c)) + a \cdot S(b) \otimes S(\alpha (c)).
\end{align}
On the other hand,
\begin{align}\label{weak-2}
(T \circ (\mu \otimes \alpha) \circ( T \otimes \mathrm{id}))( a \otimes b \otimes c) =~& ( T \circ (\mu \otimes \alpha)) (R(a) \otimes b \otimes c + a \otimes S(b) \otimes c)  \nonumber \\
=~& T( R(a) \cdot b \otimes \alpha (c) + a \cdot S(b) \otimes \alpha (c)) \nonumber \\
=~& R( R(a) \cdot b ) \otimes \alpha (c)  + R(a) \cdot b \otimes S(\alpha (c))  \\ 
& ~~~~ + R(a \cdot S(b)) \otimes \alpha (c) + a \cdot S(b) \otimes S(\alpha (c)). \nonumber
\end{align}
It follows from (\ref{weak-1}) and (\ref{weak-2}) that the right hand side pentagon of (\ref{weak-pent}) is commutative. Similarly,
\begin{align*}
((\alpha \otimes \mu ) \circ \tau)( a \otimes b \otimes c) =~& R(\alpha (c)) \otimes R(b) \cdot c + R(\alpha (a)) \otimes b \cdot S(c) + \alpha (a) \otimes S(R(b) \cdot c) + \alpha (a) \otimes S(b \cdot S(c)) \\
=~& (T \circ (\alpha \otimes \mu ) \circ (\mathrm{id} \otimes T)) (a \otimes b \otimes c). 
\end{align*}
This implies that the left-hand side pentagon of (\ref{weak-pent}) is also commutative. Hence the proof.
\end{proof}

The corresponding Hom-associative algebra structure on the Hom-vector space $(A,\alpha)$ induced from the above weak pseudotwistor coincides with (\ref{new-hom-ass}).

\section{Hom-Yang-Baxter pairs}
Let $(A, \alpha, \mu)$ be a Hom-associative algebra. An element $r =  r_{(1)} \otimes r_{(2)} \in A^{\otimes 2}$ is called $\alpha$-invariant if $(\alpha \otimes \alpha ) (r) = r$, or equivalently, $\alpha (r_{(1)}) \otimes \alpha (r_{(2)}) = r_{(1)} \otimes r_{(2)}$. As before, we omit the summation here. Let $r, s$ be two $\alpha$-invariant elements of $A^{\otimes 2}$. Define
\begin{align*}
r_{13} s_{12} =~&  r_{(1)} \cdot s_{(1)} \otimes \alpha (s_{(2)}) \otimes \alpha (r_{(2)}),\\
r_{12} s_{23} =~&  \alpha (r_{(1)}) \otimes r_{(2)} \cdot s_{(1)} \otimes \alpha (s_{(2)}),\\
r_{23} s_{13} =~&  \alpha (s_{(1)}) \otimes \alpha (r_{(1)}) \otimes r_{(2)} \cdot s_{(2)}.
\end{align*}
\vspace*{1cm}

\begin{defn}
Let $(A, \alpha, \mu)$ be a Hom-associative algebra. A pair $(r, s)$ of two $\alpha$-invariant elements of $A^{\otimes 2}$ is said to be a Hom-Yang-Baxter pair if the following conditions are hold:
\begin{align*}
\begin{cases}
r_{13} r_{12} - r_{12} r_{23} + s_{23} r_{13} = 0,\\
s_{13} r_{12} - s_{12} s_{23} + s_{23} s_{13} = 0.
\end{cases}
\end{align*}
\end{defn}

When $A$ is an associative algebra (i.e. $\alpha = \mathrm{id}$), one recover the notion of Yang-Baxter pairs introduced by Brzezi\'{n}ski \cite{brz}. On the other hand, if $r =s$, one gets the Hom-Yang-Baxter equation introduced by Yau \cite{yau}.

Let $(A, \mu)$ be an associative algebra and $(r,s)$ be an associative Yang-Baxter pair. If $\alpha : A \rightarrow A$ is an algebra morphism such that both $r, s$ are $\alpha$-invariant, then $(r,s)$ is a Hom-Yang-Baxter pair on the Hom-associative algebra $(A, \alpha, \mu_\alpha = \alpha \circ \mu)$.

\begin{exam}
Let $(A, \mu)$ be an associative algebra and $a, b \in A$ be two elements such that $a^2 = b^2 = ab = ba = 0$. Then $(r = a \otimes a, s= b \otimes b)$ is an associative Yang-Baxter pair. If $\alpha : A \rightarrow A$ is an algebra map fixing $a$ and $b$ (i.e. $\alpha (a) = a$ and $\alpha (b) = b$), then $(r,s)$ is a Hom-Yang-Baxter pair on the Hom-associative algebra $(A, \alpha, \mu_\alpha)$.

Let $(A, \mu)$ be an unital associative algebra (with unit $1$) and $a , b \in A$ be such that $b^2 = ba = 0$. Then $(r = 1 \otimes a , s = 1 \otimes b)$ is an associative Yang-Baxter pair. Thus, if $\alpha : A \rightarrow A$ is an unital algebra morphism fixing $a$ and $b$, then $(r, s)$ is a Hom-Yang-Baxter pair on the Hom-associative algebra $(A, \alpha, \mu_\alpha).$
\end{exam}

In \cite{brz} the author constructs a Rota-Baxter system from a Yang-Baxter pair in an associative algebra. In the case of Hom-associative algebras, we do not get a Rota-Baxter system on the underlying Hom-associative algebra. Rather, we get a certain twisted version of Rota-Baxter system.

\begin{defn}
Let $(A, \alpha, \mu )$ be a Hom-associative algebra. A pair $(R, S)$ of linear maps on $A$ is said to be an $(\alpha^n)$-Rota-Baxter system if $\alpha \circ R = R \circ \alpha$, $\alpha \circ S = S \circ \alpha$ and satisfies
\begin{align*}
R (\alpha^n (a)) \cdot R (\alpha^n (b)) =  R (R(a) \cdot \alpha^n (b) + \alpha^n (a) \cdot S(b)),\\
S (\alpha^n (a)) \cdot S (\alpha^n (b)) =  S (R(a) \cdot \alpha^n (b) + \alpha^n (a) \cdot S(b)).
\end{align*}
\end{defn}

Thus, an $(\alpha^0)$-Rota-Baxter system is simply a Rota-Baxter system (Definition \ref{defn-rb-s}). The notion of $(\alpha^n)$-Rota-Baxter system is a generalization of $(\alpha^n)$-Rota-Baxter operator of \cite{liu-et}.

In similar to Proposition \ref{rbs-hom-dend}, one can show that if $(R,S)$ is an $(\alpha^n)$-Rota-Baxter system on a Hom-associative algebra $(A, \alpha, \mu)$ then $(A, \alpha^{n+1}, \prec, \succ)$ is a Hom-dendriform algebra, where $a \prec b = \alpha^{n} (a) \cdot S(b) $ and $a \succ b = R(a) \cdot \alpha^n (b)$, for $a, b \in A$. As a consequence, $(A, \alpha^{n+1}, *)$ is a Hom-associative algebra and $(A, \alpha^{n+1}, \diamond)$ is a Hom-preLie algebra, where
\begin{align*}
a * b = R(a) \cdot \alpha^n (b) + \alpha^{n} (a) \cdot S(b) \quad \text{and} \quad a \diamond b = R(a) \cdot \alpha^n (b) - \alpha^{n} (a) \cdot S(b).
\end{align*}

\begin{thm}
Let $(A, \alpha, \mu)$ be a Hom-associative algebra and $(r,s)$ be a Hom-Yang-Baxter pair. Define $R , S : A \rightarrow A$ by
\begin{align*}
R(a) :=~& (r_{(1)} \cdot a ) \cdot \alpha (r_{(2)}) = \alpha (r_{(1)}) \cdot ( a \cdot r_{(2)}),\\
S(a) :=~& (s_{(1)} \cdot a ) \cdot \alpha (s_{(2)}) = \alpha (s_{(1)}) \cdot ( a \cdot s_{(2)}).
\end{align*}
Then $(R,S)$ is a $(\alpha^2)$-Rota-Baxter system on the Hom-associative algebra $(A, \alpha, \mu).$
\end{thm}

\begin{proof}
First we have
\begin{align*}
R ( \alpha (a) ) = (r_{(1)} \cdot \alpha(a) ) \cdot \alpha (r_{(2)}) = (\alpha(r_{(1)}) \cdot \alpha(a) ) \cdot \alpha^2 (r_{(2)}) = \alpha (R(a)).
\end{align*}
Similarly, we have $S(\alpha (a)) = \alpha (S(a)).$ Next, we note that the condition $r_{13} r_{12} - r_{12} r_{23} + s_{23} r_{13} = 0$, or equivalently,
\begin{align*}
 r_{(1)} \cdot \tilde{r}_{(1)} \otimes \alpha (\tilde{r}_{(2)}) \otimes \alpha (r_{(2)}) -  \alpha (r_{(1)}) \otimes r_{(2)} \cdot \tilde{r}_{(1)} \otimes \alpha (\tilde{r}_{(2)}) + \alpha (r_{(1)}) \otimes \alpha (s_{(1)}) \otimes s_{(2)} \cdot r_{(2)} = 0
\end{align*}
implies that
\begin{align}\label{pair-sys}
\big( ((r_{(1)} \cdot \tilde{r}_{(1)}) \cdot \alpha^2 (a)) \cdot \alpha^2 (\tilde{r}_{(2)}) \big) \cdot \big( \alpha^3(b) \cdot  \alpha^2 (r_{(2)}) \big)
-      \big(  ( \alpha (r_{(1)}) \cdot \alpha^2(a) ) \cdot \alpha (r_{(2)} \cdot \tilde{r}_{(1)})  \big) \cdot \big(    \alpha^3(b) \cdot \alpha^2 (\tilde{r}_{(2)}) \big)  \\
+ \big(  (\alpha (r_{(1)}) \cdot \alpha^2 (a))  \cdot \alpha^2 (s_{(1)}) \big) \cdot \big( \alpha^3 (b) \cdot \alpha (s_{(2)} \cdot r_{(2)} ) \big) = 0. \nonumber
\end{align}
We have
\begin{align*}
R (\alpha^2(a)) \cdot R (\alpha^2 (b)) 
=~& \big( ( r_{(1)} \cdot \alpha^2(a) ) \cdot \alpha (r_{(2)})  \big)   \cdot  \big(   \alpha (\tilde{r}_{(1)} ) \cdot (\alpha^2 (b) \cdot  \tilde{r}_{(2)}) \big)   \\
=~& \big(   ( \alpha (r_{(1)}) \cdot \alpha^2(a) ) \cdot \alpha^2 (r_{(2)}) \big)   \cdot  \big( \alpha (\tilde{r}_{(1)} ) \cdot (\alpha^2 (b) \cdot  \tilde{r}_{(2)}) \big)  \\
=~&  \big(   ((r_{(1)} \cdot \alpha (a) ) \cdot \alpha (r_{(2)})) \cdot \alpha (\tilde{r}_{(1)})    \big)   \cdot \big( \alpha^3 (b) \cdot  \alpha (\tilde{r}_{(2)}) \big) \\
=~& \big(    ( \alpha (r_{(1)}) \cdot \alpha^2 (a) )  \cdot ( \alpha (r_{(2)}) \cdot \tilde{r}_{(1)})  \big)   \cdot \big( \alpha^3 (b) \cdot  \alpha (\tilde{r}_{(2)}) \big) \\
=~& \big(  ( \alpha (r_{(1)}) \cdot \alpha^2(a) ) \cdot \alpha (r_{(2)} \cdot \tilde{r}_{(1)})  \big) \cdot \big(    \alpha^3(b) \cdot \alpha^2 (\tilde{r}_{(2)}) \big).
\end{align*}
Similarly,
\begin{align*}
R (R(a) \cdot \alpha^2 (b)) =~&   \alpha (r_{(1)}) \cdot\big( \big(  ((\tilde{r}_{(1)} \cdot a) \cdot \alpha (\tilde{r}_{(2)})) \cdot \alpha^2 (b) \big) \cdot r_{(2)} \big)\\
=~&  \alpha^2 (r_{(1)}) \cdot\big( \big(  ((\tilde{r}_{(1)} \cdot a) \cdot \alpha (\tilde{r}_{(2)})) \cdot \alpha^2 (b) \big) \cdot \alpha (r_{(2)}) \big) \\
=~& \alpha^2 (r_{(1)}) \cdot \big( \big( ( \alpha (\tilde{r}_{(1)})  \cdot \alpha (a) )  \cdot \alpha^2 ( \tilde{r}_{(2)})   \big) \cdot (\alpha^2 (b) \cdot r_{(2)} ) \big)\\
=~&  \big(    \alpha (r_{(1)})  \cdot \big( ( \alpha (\tilde{r}_{(1)})  \cdot \alpha (a) )  \cdot \alpha^2 ( \tilde{r}_{(2)})   \big) \big) \cdot (\alpha^3 (b) \cdot \alpha (r_{(2)}))\\
=~& \big( \big( r_{(1)} \cdot  (\alpha (\tilde{r}_{(1)}) \cdot \alpha (a)) \big) \cdot \alpha^3 (\tilde{r}_{(2)}) \big)  \cdot (\alpha^3 (b) \cdot \alpha (r_{(2)}))\\
=~& \big( \big( ( r_{(1)} \cdot \alpha (\tilde{r}_{(1)}) ) \cdot \alpha^2 (a) \big) \cdot \alpha^3 (\tilde{r}_{(2)}) \big)  \cdot (\alpha^3 (b) \cdot \alpha^2 (r_{(2)}))\\
=~& \big( \big( ( r_{(1)} \cdot \tilde{r}_{(1)} ) \cdot \alpha^2 (a) \big) \cdot \alpha^2 (\tilde{r}_{(2)}) \big)  \cdot (\alpha^3 (b) \cdot \alpha^2 (r_{(2)}))
\end{align*}
and
\begin{align*}
R (\alpha^2 (a) \cdot S(b) ) =~& \alpha (r_{(1)}) \cdot  \big( \big(  \alpha^2(a) \cdot \big( \alpha (s_{(1)}) \cdot (b \cdot s_{(2)}) \big)  \big) \cdot r_{(2)}      \big)  \\
=~& \alpha (r_{(1)} ) \cdot \big( \big(     (\alpha (a) \cdot \alpha (s_{(1)})) \cdot \alpha (b \cdot s_{(2)})   \big) \cdot r_{(2)}      \big)\\
=~& \alpha^2  (r_{(1)} ) \cdot \big( \big(     (\alpha (a) \cdot \alpha (s_{(1)})) \cdot \alpha (b \cdot s_{(2)})   \big) \cdot \alpha (r_{(2)})      \big)\\
=~& \alpha^2 (r_{(1)}) \cdot \big( (\alpha^2 (a) \cdot \alpha^2 (s_{(1)})) \cdot ( \alpha (b \cdot s_{(2)}) \cdot r_{(2)})       \big)\\
=~& \big(    \alpha (r_{(1)}) \cdot  (\alpha^2 (a) \cdot \alpha^2 (s_{(1)}) )\big) \cdot \big( \alpha^2 (b \cdot s_{(2)}) \cdot \alpha (r_{(2)})   \big)\\
=~& \big(  (r_{(1)} \cdot \alpha^2 (a) ) \cdot \alpha^3 (s_{(1)})   \big)  \cdot \big( \alpha^3 (b) \cdot (\alpha^2 (s_{(2)}) \cdot r_{(2)} )   \big) \\
=~& \big(  (\alpha (r_{(1)}) \cdot \alpha^2 (a))  \cdot \alpha^2 (s_{(1)}) \big) \cdot \big( \alpha^3 (b) \cdot \alpha (s_{(2)} \cdot r_{(2)} ) \big).
\end{align*}
Thus, it follows from (\ref{pair-sys}) that
\begin{align*}
R (\alpha^2 (a)) \cdot R (\alpha^2 (b)) = R ( R (a) \cdot \alpha^2 (b) + \alpha^2 (a) \cdot S(b))
\end{align*}
holds. In a similar way, by using the condition $s_{13} r_{12} - s_{12} s_{23} + s_{23} s_{13} = 0$, we can deduce that 
\begin{align*}
S (\alpha^2 (a)) \cdot S (\alpha^2 (b)) = S ( R (a) \cdot \alpha^2 (b) + \alpha^2 (a) \cdot S(b)).
\end{align*}
Hence the pair $(R, S)$ is an $(\alpha^2)$-Rota-Baxter system on the Hom-associative algebra $(A, \alpha, \mu)$.
\end{proof}

It follows from the above theorem that a Hom-Yang-Baxter pair $(r, s)$ induces a Hom-dendriform structure $(A, \alpha^3, \prec, \succ)$, where $a \prec b = \alpha^2 (a)  \cdot (    \alpha (s_{(1)}) \cdot (b \cdot s_{(2)}))$ and $a \succ b = ((r_{(1)} \cdot a ) \cdot \alpha (r_{(2)}) ) \cdot \alpha^2 (b).$ The associated Hom-associative and Hom-preLie algebras are respectively given by $(A, \alpha^3, *)$ and $(A, \alpha^3, \diamond)$, where
\begin{align*}
a * b =~&  ((r_{(1)} \cdot a ) \cdot \alpha (r_{(2)}) ) \cdot \alpha^2 (b) ~+~  \alpha^2 (a)  \cdot (    \alpha (s_{(1)}) \cdot (b \cdot s_{(2)})),    \\
a \diamond b =~& ((r_{(1)} \cdot a ) \cdot \alpha (r_{(2)}) ) \cdot \alpha^2 (b) ~-~  \alpha^2 (a)  \cdot (    \alpha (s_{(1)}) \cdot (b \cdot s_{(2)})).
\end{align*}

\section{Covariant Hom-bialgebras}

In this section, we introduce a notion of covariant Hom-bialgebra. This structure generalizes covariant bialgebras of Brzezi\'{n}ski \cite{brz} and infinitesimal Hom-bialgebras of Yau \cite{yau}. Given a Hom-Yang-Baxter pair, we construct a covariant Hom-bialgebra, called quasitriangular covariant Hom-bialgebra. Some characterizations of quasitriangular covariant Hom-bialgebra are given. Finally, we consider the perturbations of the coproduct in a covariant Hom-bialgebra generalizing perturbations of quasi-Hopf algebra \`{a} la Drinfel'd \cite{perturb1,perturb2}.

Let $(A, \alpha, \mu)$ be a Hom-associative algebra. A bimodule consists of a Hom-vector space $(M, \beta)$ together with action maps $l : A \otimes M \rightarrow M,~(a, m) \mapsto a \bullet m$ and $r : M \otimes A \rightarrow M,~ (m,a) \mapsto m \bullet a$ satisfying $\beta (a \bullet m) = \alpha (a) \bullet \beta (m)$, $\beta (m \bullet a) = \beta (m) \bullet \alpha (a)$ and the following bimodule identities
\begin{align*}
(a \cdot b) \bullet \beta (m) = \alpha (a) \bullet (b \bullet m),  \qquad (a \bullet m) \bullet \alpha (b) = \alpha (a) \bullet (m \bullet b), \qquad (m \bullet a) \bullet \alpha (b) = \beta (m) \bullet (a \cdot b),
\end{align*}
for $a, b \in A$ and $m \in M$. See \cite{das-hom-inf} for details. A bimodule is often denoted by $(M, \beta)$ or simply by $M$ when the action maps are understood. Any Hom-associative algebra is a bimodule over itself, called adjoint bimodule. The tensor product $M = A \otimes A$ can also be given an $A$-bimodule structure with linear map $\beta = \alpha \otimes \alpha : A \otimes A \rightarrow A \otimes A$ and action maps
\begin{align}\label{a-2-bi}
a  \bullet (b \otimes c) = \alpha (a) \cdot b \otimes \alpha (c), \qquad (b \otimes c) \bullet a = \alpha (b) \otimes c \cdot \alpha (a), ~\text{ for } a \in A, b \otimes c \in A \otimes A.
\end{align}

\begin{defn}
Let $(A, \alpha, \mu)$ be a Hom-associative algebra and $(M, \beta)$ be a bimodule over it. A linear map $d:A \rightarrow M$ is said to be a derivation on $A$ with values in the bimodule $M$ if $\beta \circ d = d \circ \alpha$ and satisfying
\begin{align*}
d(a \cdot b ) = a \bullet d(b) + d(a) \bullet b,~ \text{ for } a, b \in A.
\end{align*}
\end{defn}

Hence a linear map $d: A \rightarrow A \otimes A$ is a derivation on $A$ with values in the bimodule $A \otimes A$ if and only if $\alpha^{\otimes 2} \circ d = d \circ \alpha$ and $D$ satisfies
\begin{align*}
d \circ \mu = (\mu \otimes \alpha ) \circ ( \alpha \otimes d) + ( \alpha \otimes \mu) \circ (d \otimes \alpha ).
\end{align*}

Let $(A, \alpha, \mu)$ be a Hom-associative algebra and $\delta_1, \delta_2 : A \rightarrow A \otimes A$ be two derivations on $A$ with values in the $A$-bimodule $A \otimes A$.

\begin{defn}
A linear map $D : A \rightarrow A \otimes A$ is said to be a covariant derivation with respect to $(\delta_1, \delta_2)$ if $D$ satisfies $\alpha^{\otimes 2} \circ D = D \circ \alpha $ and
\begin{align*}
D \circ \mu =~& (\mu \otimes \alpha ) \circ (\alpha \otimes \delta_1) + (\alpha \otimes \mu) \circ ( D \otimes \alpha),\\
D \circ \mu =~& (\mu \otimes \alpha ) \circ (\alpha \otimes D)  + (\alpha \otimes \mu ) \circ (\delta_2 \otimes \alpha).
\end{align*}
\end{defn}
It follows that any derivation $d$ is a covariant derivation with respect to $(d,d)$.
We are now in a position to define covariant Hom-bialgebras.

\begin{defn}
A covariant Hom-bialgebra consists of a tuple $(A, \alpha, \mu, \triangle, \delta_1, \delta_2)$ in which
\begin{itemize}
\item[(i)] $(A, \alpha, \mu)$ is a Hom-associative algebra;
\item[(ii)] $(A, \alpha, \triangle)$ is a Hom-coassociative coalgebra;
\item[(iii)] the linear maps $\delta_1, \delta_2 :A \rightarrow A \otimes A$ are derivations on $A$ with values in the $A$-bimodule $A \otimes A$;
\item[(iv)] the map $\triangle : A \rightarrow A \otimes A$ is a covariant derivation with respect to $(\delta_1, \delta_2)$.
\end{itemize}
\end{defn}

If $\alpha = \mathrm{id}$, one recover the notion of covariant bialgebras introduced by Brzezi\'{n}ski \cite{brz}.
On the other hand, if $\triangle = \delta_1= \delta_2$, one recover infinitesimal Hom-bialgebras defined in \cite{yau}. Further, if $\alpha = \mathrm{id}$, we get infinitesimal bialgebras \cite{joni-rota,aguiar}.

\begin{prop}\label{coin}
Let $(A, \mu, \triangle, \delta_1, \delta_2)$ be a covariant bialgebra and $\alpha : A \rightarrow A$ be a morphism of covariant bialgebras. Then $(A, \alpha, \mu_\alpha, \triangle_\alpha, (\delta_1)_\alpha, (\delta_2)_\alpha)$ is a covariant Hom-bialgebra, where $\triangle_\alpha = \triangle \circ \alpha$, $(\delta_1)_\alpha = \delta_1 \circ \alpha$ and $(\delta_2)_\alpha = \delta_2 \circ \alpha$.
\end{prop}

\begin{proof}
Since $\alpha : (A, \mu) \rightarrow (A, \mu)$ is an algebra morphism and $\alpha : (A, \triangle) \rightarrow (A, \triangle)$ is a coassociative coalgebra morphism, it follows that $(A, \alpha, \mu_\alpha)$ is a Hom-associative algebra and $(A, \alpha, \triangle_\alpha)$ is a Hom-coassociative coalgebra. Moreover, for $i=1,2$, we have
\begin{align*}
(\delta_i)_\alpha \circ \mu_\alpha =~& (\alpha^2)^{\otimes 2} \circ \delta_i \circ \mu \\
=~& (\alpha^2)^{\otimes 2} \circ (\mathrm{id} \otimes \mu) \circ (\delta_i \otimes \mathrm{id}) + (\alpha^2)^{\otimes 2} \circ (\mu \otimes \mathrm{id}) \circ (\mathrm{id} \otimes \delta_i ) \\
=~& (\alpha^2 \otimes \mu_\alpha \circ \alpha^{\otimes 2}) \circ (\delta_i \otimes \mathrm{id}) + (\mu_\alpha \circ \alpha^{\otimes 2} \otimes \alpha^2) \circ (\mathrm{id} \otimes \delta_i) \\
=~& (\alpha \otimes \mu_\alpha ) \circ \alpha^{\otimes 3} \circ (\delta_i \otimes \mathrm{id}) + (\mu_\alpha \otimes \alpha) \circ \alpha^{\otimes 3} \circ (\mathrm{id} \otimes \delta_i) \\
=~& (\mu_\alpha \otimes \alpha ) \circ (\alpha \otimes (\delta_i)_\alpha ) + (\alpha \otimes \mu_\alpha) \circ ((\delta_i)_\alpha \otimes \alpha).
\end{align*}
This shows that $(\delta_1)\alpha$ and $(\delta_2)_\alpha$ are both derivations. Finally, 
\begin{align*}
\triangle_\alpha \circ \mu_\alpha =~& (\alpha^2)^{\otimes 2} \circ \triangle \circ \mu \\
=~& (\alpha^2)^{\otimes 2} \circ (\mu \otimes \mathrm{id}) \circ (\mathrm{id} \otimes \delta_1 ) + (\alpha^2)^{\otimes 2} \circ (\mathrm{id} \otimes \mu) \circ (\triangle \otimes \mathrm{id}) \\
=~& (\mu_\alpha \circ \alpha^{\otimes 2} \otimes \alpha^2) \circ (\mathrm{id} \otimes \delta_1) + (\alpha^2 \otimes \mu_\alpha \circ \alpha^{\otimes 2}) \circ (\triangle \otimes \mathrm{id})  \\
=~& (\mu_\alpha \otimes \alpha) \circ \alpha^{\otimes 3} \circ (\mathrm{id} \otimes \delta_1) + (\alpha \otimes \mu_\alpha ) \circ \alpha^{\otimes 3} \circ (\triangle \otimes \mathrm{id})  \\
=~& (\mu_\alpha \otimes \alpha ) \circ (\alpha \otimes (\delta_1)_\alpha ) + (\alpha \otimes \mu_\alpha) \circ ((\triangle)_\alpha \otimes \alpha).
\end{align*}
Similarly, we can show that $\triangle_\alpha \circ \mu_\alpha = (\mu_\alpha \otimes \alpha ) \circ (\alpha \otimes \triangle_\alpha)  + (\alpha \otimes \mu_\alpha ) \circ ((\delta_2)_\alpha \otimes \alpha).
$ Therefore, $\triangle_\alpha$ is a covariant derivation with respect to $((\delta_1)_\alpha , (\delta_2)_\alpha)$. Hence $(A, \alpha, \mu_\alpha, \triangle_\alpha,(\delta_1)_\alpha , (\delta_2)_\alpha) $ is a covariant Hom-bialgebra.
\end{proof}

The covariant Hom-bialgebra $(A, \alpha, \mu_\alpha, \triangle_\alpha, (\delta_1)_\alpha, (\delta_2)_\alpha)$ obtained in Proposition \ref{coin} is called `induced by composition'.

In the following, we introduce another type of Hom-bialgebras dual to covariant Hom-bialgebras. Before that, we require some more terminologies.

Let $(C, \alpha, \triangle)$ be a Hom-coassociative coalgebra. A bicomodule over it consists of a Hom-vector space $(M, \beta)$ together with coaction maps $\triangle^l : M \rightarrow C \otimes M$ and $\triangle^r : M \rightarrow M \otimes C$ satisfying $(\alpha \otimes \beta) \circ \triangle^l = \triangle^l \circ \beta$, $(\beta \otimes \alpha ) \circ \triangle^r = \triangle^r \circ \beta$ and
\begin{align*}
(\alpha \otimes \triangle^l ) \circ \triangle^l = (\triangle \otimes \beta ) \circ \triangle^l, \qquad
(\alpha \otimes \triangle^r ) \circ \triangle^l = (\triangle^l \otimes \alpha ) \circ \triangle^r, \qquad
(\beta \otimes \triangle ) \circ \triangle^r = (\triangle^r \otimes \alpha ) \circ \triangle^r.
\end{align*}
It follows that any Hom-coassociative coalgebra is a bicomodule over itself. The tensor product $M = C \otimes C$ with the linear map $\beta = \alpha \otimes \alpha : C \otimes C \rightarrow C \otimes C$ and coaction maps $\triangle^l : C \otimes C \otimes (C \otimes C)$, $\triangle^r : C \otimes C \rightarrow (C \otimes C) \otimes C$ defined by
\begin{align*}
\triangle^l  =  ((\alpha \otimes \mathrm{id}) \circ \triangle) \otimes \alpha,  \qquad \triangle^r = \alpha \otimes ((\mathrm{id} \otimes \alpha) \circ \triangle )
\end{align*}
is a bicomodule over $C$.

Given a Hom-coassociative coalgebra $(C, \alpha, \triangle)$ and a bicomodule $(M, \beta)$, a linear map $D : M \rightarrow C$ is a coderivation on $C$ with coefficients in $M$ is $\alpha \circ D = D \circ \beta$ and satisfies
\begin{align*}
\triangle \circ D = (\mathrm{id}_C \otimes D) \circ \triangle^l + (D \otimes \mathrm{id}_C) \circ \triangle^r.
\end{align*}
It follows that a linear map $\partial : C \otimes C \rightarrow C$ is a coderivation on $C$ with coefficients in the bicomodule $C \otimes C$ is $\alpha \otimes \partial = \partial \circ (\alpha \otimes \alpha)$ and
\begin{align*}
\triangle \circ \partial = (\alpha \otimes \partial) \circ  (\triangle \otimes \alpha ) + (\partial \otimes \alpha ) \circ (\alpha \otimes \triangle).
\end{align*}
Let $\partial_1, \partial_2$ be two coderivations on $C$ with coefficients in $C \otimes C$. Then a linear map $\mu : C \otimes C \rightarrow C$ is called a covariant coderivation with respect to $(\partial_1, \partial_2)$ if
\begin{align*}
\triangle \circ \mu = (\alpha \otimes \partial_1) \circ (\triangle \otimes \alpha ) + (\mu \otimes \alpha ) \circ (\alpha \otimes \triangle), \\
\triangle \circ \mu = (\alpha \otimes \mu) \circ (\triangle \otimes \alpha ) + (\partial_2 \otimes \alpha ) \circ (\alpha \otimes \triangle).
\end{align*}

\begin{defn}
A dual covariant Hom-bialgebra is a tuple $(A, \alpha, \mu, \triangle, \partial_1, \partial_2)$ in which $(A, \alpha, \mu)$ is a Hom-associative algebra, $(A, \alpha, \triangle)$ is a Hom-coassociative coalgebra, the maps $\partial_1, \partial_2 : A \otimes A \rightarrow A$ are coderivations on $A$ with coefficients in the bicomodule $A \otimes A$ such that $\mu : A \otimes A \rightarrow A$ is a covariant coderivation with respect to $(\partial_1, \partial_2).$
\end{defn}

\begin{prop}
Let $(A, \alpha, \mu, \triangle, \delta_1, \delta_2)$ be a finite dimensional covariant Hom-bialgebra. Then the tuple $(A^*, \alpha^*, \triangle^*, \mu^*, \delta_1^*, \delta_2^*)$ is a dual covariant Hom-bialgebra, where
\begin{align*}
\langle \triangle^* (\phi \otimes \psi), a \rangle = \langle \phi \otimes \psi, \triangle (a) \rangle,~~~
\langle \mu^* (\phi), a \otimes b \rangle = \langle \phi, \mu (a \otimes b) \rangle, ~~~
\langle \delta_i^* (\phi \otimes \psi), a \rangle = \langle \phi \otimes \psi, \delta_i (a) \rangle, \text{ for } i =1,2.
\end{align*}
\end{prop}

\begin{proof}
Since $(A, \alpha, \triangle)$ is a Hom-coassociative coalgebra, we have $(A^*, \alpha^*, \triangle^*)$ is a Hom-associative algebra. On the other hand, $(A, \alpha, \mu)$ is a Hom-associative algebra and $A$ is finite dimensional implies that $(A^*, \alpha^*, \mu^*)$ is a Hom-coassociative coalgebra \cite{makh-sil}. Next, for $i=1,2$, we observe that
\begin{align*}
\mu^* \circ \delta_i^* = (\delta_i \circ \mu)^* =~& ((\mu \otimes \alpha) \circ (\alpha \otimes \delta_i))^* + ((\alpha \otimes \mu) \circ (\delta_i \otimes \alpha))^* \\
=~& (\alpha^* \otimes \delta_i^*) \circ (\mu^* \otimes \alpha^*) + (\delta_i^* \otimes \alpha^*) \circ (\alpha^* \otimes \mu^*).
\end{align*}
This shows that $\delta_1^*, \delta_2^*$ are coderivations on the Hom-coassociative coalgebra $A^*$ with values in the bicomodule $A^* \otimes A^*$. Finally, it remains to prove that $\triangle^* : A^* \otimes A^* \rightarrow A^*$ is a covariant coderivation with respect to $(\delta^*_1, \delta^*_2)$. Note that
\begin{align*}
\mu^* \circ \triangle^* = (\triangle \circ \mu)^* =~& ((\mu \otimes \alpha) \circ (\alpha \otimes \delta_1))^* + ((\alpha \otimes \mu) \circ (\triangle \otimes \alpha))^* \\
=~& (\alpha^* \otimes \delta_1^*) \circ (\mu^* \otimes \alpha^*) + (\triangle^* \otimes \alpha^*) \circ (\alpha^* \otimes \mu^*).
\end{align*}
Similarly, we can show that $ \mu^* \circ \triangle^* = (\alpha^* \otimes \triangle^*) \circ (\mu^* \otimes \alpha^*) + (\delta_2^* \otimes \alpha^*) \circ (\alpha^* \otimes \mu^*).$ Hence the proof.
\end{proof}

Given a Hom-associative algebra $(A, \alpha, \mu)$ and a Hom-Yang-Baxter pair $(r,s)$, we will construct a covariant Hom-bialgebra. We need some more notations. Let $(A, \alpha, \mu)$ be a Hom-associative algebra. Then $M = A^{\otimes n}$ with the linear map $\beta = \alpha^{\otimes n} : A^{\otimes n} \rightarrow A^{\otimes n}$ can be given an $A$-bimodule structure with actions
\begin{align*}
a \bullet (b_1 \otimes \cdots \otimes b_n) =~& \alpha (a) \cdot b_1 \otimes \alpha (b_2) \otimes \cdots \otimes \alpha (b_n), \\
(b_1 \otimes \cdots \otimes b_n) \bullet a =~& \alpha (b_1) \otimes \cdots \otimes \alpha(b_{n-1}) \otimes b_n \cdot \alpha (a).
\end{align*}
This generalizes the $A$-bimodule structure on $A^{\otimes 2}$ given in (\ref{a-2-bi}). Moreover, for any $a \in A$ and $r = r_{(1)} \otimes r_{(2)} \in A^{\otimes 2}$, we define $ar, ra \in A^{\otimes 2}$ by
\begin{align*}
ar = a \cdot r_{(1)} \otimes \alpha (r_{(2)}) ~~~~~ \text{ and } ~~~~~ ra = \alpha (r_{(1)}) \otimes r_{(2)} \cdot a.
\end{align*}

\begin{thm}\label{quasi-hom-co}
Let $(A, \alpha, \mu)$ be a Hom-associative algebra and $r, s \in A^{\otimes 2}$ be two $\alpha$-invariant elements of $A^{\otimes 2}$. Define maps $\triangle', \delta_r , \delta_s : A \rightarrow A \otimes A$ by
\begin{align}\label{all-maps}
\triangle' (a) = ar -sa, \qquad \delta_r (a) = ar -ra, \qquad \delta_s(a) = as -sa.
\end{align}
Then $(A, \alpha, \mu, \triangle', \delta_r, \delta_s)$ is a covariant Hom-bialgebra if and only if
\begin{align}\label{quasi-rs}
a \bullet ( r_{13}r_{12} - r_{12} r_{23} + s_{23} r_{13} ) = (s_{13} r_{12} - s_{12} s_{23} + s_{23}s_{13}) \bullet a.
\end{align}
\end{thm}

\begin{proof}
First we have
\begin{align*}
\triangle' (\alpha (a) )  =~& \alpha (a) \cdot r_{(1)} \otimes \alpha (r_{(2)}) - \alpha (s_{(1)}) \otimes s_{(2)} \cdot \alpha (a) \\
=~& \alpha (a) \cdot \alpha (r_{(1)}) \otimes \alpha^2 (r_{(2)}) - \alpha^2 (s_{(1)}) \otimes \alpha(s_{(2)}) \cdot \alpha (a) \\
=~& (\alpha\otimes \alpha) (a \cdot r_{(1)} \otimes \alpha (r_{(2)}) - \alpha (s_{(1)}) \otimes s_{(2)} \cdot a)  = (\alpha\otimes \alpha) \triangle' (a).
\end{align*}
This shows that $\triangle'$ commutes with $\alpha$. Similarly, $\delta_r$ and $\delta_s$ also commutes with $\alpha$. Next, we claim that $\delta_r$ is a derivation. We observe that
\begin{align}\label{delta-r-der}
\delta_r \circ \mu ( a \otimes b) = \delta_r (a \cdot b ) =~&  (a \cdot b) \cdot r_{(1)} \otimes \alpha (r_{(2)}) - \alpha (r_{(1)}) \otimes r_{(2)} \cdot (a \cdot b)  \nonumber \\
=~& (a \cdot b) \cdot \alpha (r_{(1)}) \otimes \alpha^2 (r_{(2)}) - \alpha^2 (r_{(1)}) \otimes \alpha(r_{(2)}) \cdot (a \cdot b)  \nonumber  \\
=~& \alpha (a) \cdot (b \cdot r_{(1)}) \otimes \alpha^2 (r_{(2)}) - \alpha^2 (r_{(1)}) \otimes (r_{(2)} \cdot a) \cdot \alpha (b).
\end{align}
On the other hand,
\begin{align}\label{delta-r-der-2}
&((\mu \otimes \alpha) \circ (\alpha \otimes \delta_r ) + (\alpha \otimes \mu) \circ (\delta_r \otimes \alpha))(a \otimes b)  \nonumber \\
&= (\mu \otimes \alpha ) \big( \alpha (a) \otimes b \cdot r_{(1)} \otimes \alpha (r_{(2)}) - \alpha (a) \otimes \alpha (r_1) \otimes r_{(2)} \cdot b \big)  \nonumber \\
&~~~+ (\alpha \otimes \mu) \big(   a \cdot r_{(1)} \otimes \alpha (r_{(2)}) \otimes \alpha (b)  - \alpha (r_{(1)}) \otimes r_{(2)} \cdot a \otimes \alpha (b) \big)  \nonumber  \\
&= \alpha(a) \cdot (b \cdot r_{(1)}) \otimes \alpha^2 (r_{(2)}) - \cancel{\alpha (a) \cdot \alpha (r_{(1)}) \otimes \alpha (r_{(2)} \cdot b)} \\
&~~~ + \cancel{\alpha (a \cdot r_{(1)}) \otimes \alpha (r_{(2)}) \cdot \alpha (b)} - \alpha^2 (r_{(1)}) \otimes (r_{(2)} \cdot a) \cdot \alpha (b).  \nonumber 
\end{align}
It follows from (\ref{delta-r-der}) and (\ref{delta-r-der-2}) that $\delta_r$ is a derivation. Replacing $r$ by $s$, we get that $\delta_s$ is a derivation. Thus, to show that $(A, \alpha, \mu, \triangle', \delta_r, \delta_s)$ is a covariant Hom-bialgebra, we need to verify that $\triangle'$ is Hom-coassociative. In the next, we will see that the Hom-coassociativity of $\triangle'$ is equivalent to the condition (\ref{quasi-rs}) that proves the theorem. For $r = r_{(1)} \otimes r_{(2)}$ and $s = s_{(1)} \otimes s_{(2)}$, we have
\begin{align}\label{coas-1}
( \triangle' \otimes \alpha) (\triangle' a) =~& (\triangle' \otimes \alpha) \big( a \cdot r_{(1)} \otimes \alpha (r_{(2)}) -  \alpha (s_{(1)}) \otimes s_{(2)} \cdot a \big)  \nonumber \\
=~& \triangle' ( a \cdot r_{(1)}) \otimes \alpha^2 (r_{(2)}) - \alpha^{\otimes 2} (\triangle' (s_{(1)})) \otimes \alpha (s_{(2)}) \cdot \alpha (a)  \nonumber \\
=~& \alpha (a) \cdot (r_{(1)} \cdot \tilde{r}_{(1)} ) \otimes \alpha^2 ( \tilde{r}_{(2)}) \otimes \alpha^2 (r_{(2)}) - \alpha^2 (s_{(1)}) \otimes ( s_{(2)} \cdot a) \cdot \alpha (r_{(1)}) \otimes \alpha^2 (r_{(2)})  \nonumber  \\
 &- \alpha ( s_{(1)} \cdot \tilde{r}_{(1)}) \otimes \alpha^2 (\tilde{r}_{(2)}) \otimes \alpha (s_{(2)}) \cdot \alpha (a) + \alpha^2 (\tilde{s}_{(1)}) \otimes \alpha (\tilde{s}_{(2)} \cdot s_{(1)}) \otimes \alpha (s_{(2)}) \cdot \alpha (a)  \nonumber \\
 =~& a \bullet (r_{13} r_{12}) - \alpha^2 ( s_{(1)} ) \otimes (s_{(2)} \cdot a) \cdot \alpha (r_{(1)}) \otimes \alpha^2 (r_{(2)}) - (s_{13} r_{12}) \bullet a + (s_{12} s_{23}) \bullet a.
\end{align}
Similarly,
\begin{align}\label{coas-2}
(\alpha \otimes \triangle') (\triangle' a) =~& (\alpha \otimes \triangle')  \big( a \cdot r_{(1)} \otimes \alpha (r_{(2)}) - \alpha (s_{(1)}) \otimes s_{(2)} \cdot a   \big) \nonumber \\
=~& \alpha (a) \cdot \alpha (r_{(1)}) \otimes \alpha^{\otimes 2} (\triangle' (r_{(2)})) - \alpha^2 (s_{(1)}) \otimes \triangle' (s_{(2)} \cdot a)  \nonumber \\
=~& \alpha (a) \cdot \alpha (r_{(1)}) \otimes \alpha (r_{(2)} \cdot \tilde{r}_{(1)}) \otimes \alpha^2 (\tilde{r}_{(2)}) - \alpha (a) \cdot \alpha (r_{(1)}) \otimes \alpha^2 (  {s}_{(1)} ) \otimes \alpha ({s}_{(2)} \cdot r_{(2)})  \nonumber \\
&- \alpha^2 (s_{(1)}) \otimes \alpha (s_{(2)}) \cdot ( a \cdot r_{(1)}) \otimes \alpha^2 ( r_{(2)}) + \alpha^2 (s_{(1)}) \otimes \alpha^2 (\tilde{s}_{(1)}) \otimes (\tilde{s}_{(2)} \cdot s_{(2)}) \cdot \alpha(a) \nonumber  \\
=~& a \bullet (r_{12} r_{23} ) - a \bullet (s_{23}r_{13}) - \alpha^2 (s_{(1)} ) \otimes (s_{(2)} \cdot a) \cdot \alpha ( r_{(1)}) \otimes \alpha^2 (r_{(2)}) + (s_{23} s_{13}) \bullet a.
\end{align}
It follows from (\ref{coas-1}) and (\ref{coas-2}) that 
\begin{align}\label{diff}
(\triangle' \otimes \alpha ) (\triangle' a) - (\alpha \otimes \triangle') (\triangle' a) = 
a \bullet ( r_{13}r_{12} - r_{12} r_{23} + s_{23} r_{13} ) - (s_{13} r_{12} - s_{12} s_{23} + s_{23}s_{13}) \bullet a.
\end{align}
This shows that $\triangle'$ is Hom-coassociative if and only if (\ref{quasi-rs}) holds. This completes the proof.
\end{proof}

It follows from the above theorem that a Hom-Yang-Baxter pair $(r,s)$ on a Hom-associative algebra $(A, \alpha, \mu)$ induces a covariant Hom-bialgebra $(A, \alpha, \mu, \triangle', \delta_r, \delta_s)$, called quasitriangular covariant bialgebra.

When $\alpha = \mathrm{id}$, one recover the notion of quasitriangular covariant bialgebra introduced in \cite{brz}. On the other hand, if $r=s$ (i.e. $r$ is a solution of Hom-Yang-Baxter equation), one get quasitriangular infinitesimal Hom-bialgebra introduced in \cite{yau}.

The following theorem gives some characterizations of quasitriangular covariant Hom-bialgebras.

\begin{thm}
Let $(A, \alpha, \mu)$ be a Hom-associative algebra and $r, s$ be two $\alpha$-invariant elements of $A^{\otimes 2}$. Let $\triangle', \delta_r, \delta_s$ be defined by (\ref{all-maps}). Then the followings are equivalent:

\begin{itemize}
\item[(i)] $(A, \alpha, \mu,\triangle', \delta_r, \delta_s)$ is a quasitriangular covariant Hom-bialgebra;
\item[(ii)] $(\alpha \otimes \triangle')(r) = r_{13}r_{12}$ and $(\triangle' \otimes \alpha )(s) = - s_{23} s_{13}$;
\item[(iii)] The diagrams
\begin{align}\label{cov-diag}
\xymatrix{
& A^* \otimes A^* \ar[d]_{\triangle'^*} \ar[r]^{\rho_1^r \otimes \rho_1^r}  & A \otimes A \ar[d]^{\mu^{\mathrm{op}}}& & A^* \otimes A^* \ar[d]_{\triangle'^*} \ar[r]^{\lambda_1^s \otimes \lambda_1^s} & A \otimes A \ar[d]^{- \mu^{\mathrm{op}}} \\
 & A^* \ar[r]_{\rho^r_2}  & A & & A^* \ar[r]_{\lambda^s_2} & A\\
}
\end{align}
are commutative. Here $\rho_1^r, \rho_2^r, \lambda_1^s, \lambda_2^s : A^* \rightarrow A$ are maps defined respectively by
\begin{align*}
\rho_1^r (\phi) = r_{(1)} \langle \phi, \alpha (r_{(2)}) \rangle,  \qquad \rho_2^r (\phi) = \alpha (r_{(1)}) \langle \phi, r_{(2)} \rangle, \\
\lambda_1^s (\phi ) = \langle \phi, \alpha (s_{(1)}) \rangle,  \qquad \lambda_2^s (\phi) = \langle \phi, s_{(1)}\rangle \alpha (s_{(2)}).
\end{align*}
\end{itemize}
\end{thm}

\begin{proof}
(i) $\Longleftrightarrow$ (ii). By definition, the tuple $(A, \alpha, \mu, \triangle', \delta_r, \delta_s)$ is a quasitriangular covariant Hom-bialgebra if and only if $(r,s)$ is a Hom-Yang-Baxter pair. This holds if and only if
\begin{align*}
(\alpha \otimes \triangle')(r) =~& \alpha (r_{(1)} ) \otimes \triangle' (r_{(2)}) \\
=~& \alpha (r_{(1)}) \otimes r_{(2)} \cdot \tilde{r}_{(1)} \otimes \alpha (\tilde{r}_{(2)}) - \alpha (r_{(1)}) \otimes \alpha (s_{(1)}) \otimes s_{(2)} \cdot r_{(2)} \\
=~& r_{12} r_{23} -s_{23}r_{13} = r_{13} r_{12},
\end{align*}
and
\begin{align*}
(\triangle' \otimes \alpha )(s) =~& \triangle' (s_{(1)}) \otimes \alpha (s_{(2)}) \\
 =~& s_{(1)} \cdot r_{(1)} \otimes \alpha (r_{(2)}) \otimes \alpha (s_{(2)}) - \alpha ( \tilde{s}_{(1)} ) \otimes \tilde{s}_{(2)} \cdot s_{(1)} \otimes \alpha (s_{(2)}) \\
 =~& s_{13} r_{12} - s_{12} s_{23} = - s_{23} s_{13}.
\end{align*}

(i) $\Longleftrightarrow$ (iii). Note that $r_{13} r_{12} -r_{12} r_{23} + s_{23} r_{13} = 0$ and $ s_{13} r_{12} -s_{12} s_{23} + s_{23} s_{13}= 0 $ if and only if
\begin{align*}
 \langle \mathrm{id}_A \otimes \phi \otimes \psi,  r_{13} r_{12} -r_{12} r_{23} + s_{23} r_{13}  \rangle = 0 \quad \text{ and } \quad \langle \phi \otimes \psi \otimes \mathrm{id}_A ,  s_{13} r_{12} -s_{12} s_{23} + s_{23} s_{13}  \rangle = 0,
\end{align*}
for all $\phi, \psi \in A^*.$ First we have
\begin{align*}
r_{13} r_{12} -r_{12} r_{23} + s_{23} r_{13} =~& r_{13} r_{12} - (\alpha \otimes \triangle')(r) \\
=~& r_{(1)} \cdot \tilde{r}_{(1)} \otimes \alpha (\tilde{r}_{(2)} ) \otimes \alpha (r_{(2)}) - \alpha (r_{(1)}) \otimes \triangle' (r_{(2)}). 
\end{align*}
Hence 
\begin{align*}
&\langle \mathrm{id}_A \otimes \phi \otimes \psi, r_{13} r_{12} -r_{12} r_{23} + s_{23} r_{13} \rangle \\
&= r_{(1)} \cdot \tilde{r}_{(1)} \langle \phi, \alpha (\tilde{r}_{(2)}) \rangle \langle \psi, \alpha (r_{(2)}) \rangle - \alpha (r_{(1)}) \langle \triangle'^* (\phi \otimes \psi), r_{(2)} \rangle \\
&= \mu^{\mathrm{op}} (( \rho_1^r)^{\otimes 2} (\phi \otimes \psi)) - \rho_2^r ( \triangle'^* (\phi \otimes \psi)).
\end{align*}
Similarly, 
\begin{align*}
s_{13} r_{12} -s_{12} s_{23} + s_{23} s_{13} = (\triangle' \otimes \alpha )(s) + s_{23} s_{13} 
= \triangle' (s_{(1)}) \otimes \alpha (s_{(2)}) + \alpha (s_{(1)}) \otimes \alpha (\tilde{s}_{(1)}) \otimes \tilde{s}_{(2)} \cdot s_{(2)}.
\end{align*}
Hence
\begin{align*}
&\langle \phi \otimes \psi \otimes \mathrm{id}_A ,  s_{13} r_{12} -s_{12} s_{23} + s_{23} s_{13}  \rangle \\
&= \langle \triangle'^* (\phi \otimes \psi), s_{(1)} \rangle \alpha (s_{(2)}) + \langle \phi, \alpha (s_{(1)}) \rangle \langle \psi, \alpha (\tilde{s}_{(1)}) \rangle \tilde{s}_{(2)} \cdot s_{(2)} \\
&= \lambda^s_2 (\triangle'^* (\phi \otimes \psi)) + \mu^{\mathrm{op}} ((\lambda_1^s)^{\otimes 2} (\phi \otimes \mu)).
\end{align*}
Hence the result follows.
\end{proof}

Let $(A, \alpha, \mu, \triangle, \delta_1, \delta_2)$ be a covariant Hom-bialgebra.
Next we consider perturbations of the coproduct $\triangle $ and derivations  $\delta_1, \delta_2$ by a suitable pair $(r,s)$ of $\alpha$-invariant elements of $A^{\otimes 2}$. These perturbations generalize the perturbation theory of quasi-Hopf algebras studied by Drinfel'd \cite{perturb1,perturb2}. The following result also generalizes the result of Yau for infinitesimal Hom-bialgebras \cite{yau}.

\begin{thm}
Let $(A, \alpha, \mu, \triangle, \delta_1, \delta_2)$ be a covariant Hom-bialgebra and $(r,s)$ be a pair of $\alpha$-invariant elements of $A^{\otimes 2}$. Then $(A, \alpha, \mu, \triangle + \triangle' , \delta_1 + \delta_r , \delta_2 + \delta_s)$ is also a covariant Hom-bialgebra if and only if
\begin{align}\label{perturb}
a \bullet \big( (\alpha \otimes \triangle)^{-} (r) - ( r_{13}r_{12} - r_{12} r_{23} + s_{23} r_{13} )  \big) -~& \big(  (\alpha \otimes \triangle)^{-} (s) -  ( s_{13} r_{12} - s_{12} s_{23} + s_{23}s_{13})  \big) \bullet a \\
&= s_{23} \triangle(a)_{13} + \triangle(a)_{13} r_{12}. \nonumber
\end{align}
\end{thm}

\begin{proof}
First note that each $\triangle', \delta_r, \delta_s$ commutes with $\alpha$. Therefore, the sums $\triangle + \triangle' , \delta_1 + \delta_r, \delta_2 + \delta_s$ also commute with $\alpha$. Likewise, $\delta_r$ and $\delta_s$ are derivations, so the sums $\delta_1 + \delta_r$ and $\delta_2 + \delta_s$. We claim that $\triangle + \triangle'$ is a covariant derivation with respect to $(\delta_1 + \delta_r , \delta_2 + \delta_s)$. We have
\begin{align*}
(\triangle + \triangle') \circ \mu =~& \triangle \circ \mu + \triangle' \circ \mu  \\
=~& (\mu \otimes \alpha) \circ (\alpha \otimes \triangle) + (\alpha \otimes \mu) \circ (\delta_2 \otimes \alpha) 
+ (\mu \otimes \alpha ) \circ (\alpha \otimes \triangle') + (\alpha \otimes \mu ) \circ (\delta_s \otimes \alpha ) \\
=~& (\mu \otimes \alpha)\circ (\alpha \otimes (\triangle + \triangle')) + (\alpha \otimes \mu ) \circ ((\delta_2 + \delta_s) \otimes \alpha ).
\end{align*} 
Similarly, we can show that $(\triangle + \triangle') \circ \mu = (\mu \otimes \alpha ) \circ ( \alpha \otimes (\delta_1 + \delta_r)) + (\alpha \otimes \mu) \circ ((\triangle + \triangle') \otimes \alpha ).$ Hence we proved our claim. Thus, we are left with the Hom-coassociativity of $\triangle + \triangle'$. We show that $\triangle + \triangle'$ is Hom-coassociative if and only if (\ref{perturb}) holds. First observe that
\begin{align}\label{perturb-1}
(\alpha \otimes (\triangle + \triangle')) \circ ( \triangle + \triangle') = (\alpha \otimes \triangle ) \circ \triangle + (\alpha \otimes \triangle) \circ \triangle' + (\alpha \otimes \triangle') \circ \triangle + (\alpha \otimes \triangle') \circ \triangle'.
\end{align}
Similarly,
\begin{align}\label{perturb-2}
(\triangle + \triangle') \otimes \alpha ) \circ (\triangle + \triangle') = (\triangle \otimes \alpha ) \circ \triangle + (\triangle \otimes \alpha ) \circ \triangle' + (\triangle' \otimes \alpha ) \circ \triangle + (\triangle' \otimes \alpha ) \circ \triangle'.
\end{align}
Since $\triangle $ is Hom-coassociative, we have $(\alpha \otimes \triangle) \circ \triangle = (\triangle \otimes \alpha ) \circ \triangle$. Therefore, it follows from (\ref{perturb-1}) and (\ref{perturb-2}) that $(\triangle + \triangle')$ is Hom-coassociative if and only if
\begin{align*}
(\alpha \otimes \triangle - \triangle \otimes \alpha ) \circ \triangle' + (\alpha \otimes \triangle' - \triangle' \otimes \alpha ) \circ \triangle = (\triangle' \otimes \alpha - \alpha \otimes \triangle') \circ \triangle',
\end{align*}
or equivalently,
\begin{align*}
(\alpha \otimes \triangle)^{-} \circ \triangle'(a) ~+~& (\alpha \otimes \triangle')^{-} \circ \triangle (a) \\
=~& a \bullet ( r_{13}r_{12} - r_{12} r_{23} + s_{23} r_{13}) - ( s_{13} r_{12} - s_{12} s_{23} + s_{23}s_{13} ) \bullet a  \quad (\text{by } (\ref{diff})).
\end{align*}
For $r = r_{(1)} \otimes r_{(2)}$ and $s = s_{(1)} \otimes s_{(2)}$, we have
\begin{align}
&(\alpha \otimes \triangle)^{-} \circ \triangle'(a)  \nonumber \\
&=(\alpha \otimes \triangle - \triangle \otimes \alpha ) \big( a \cdot r_{(1)} \otimes \alpha (r_{(2)}) - \alpha (s_{(1)}) \otimes s_{(2)} \cdot a \big) \nonumber \\
&= \alpha (a) \cdot \alpha ( r_{(1)}) \otimes \triangle (\alpha (r_{(2)})) - \alpha^2 (s_{(1)}) \otimes \triangle (s_{(2)} \cdot a) - \triangle (a \cdot r_{(1)}) \otimes \alpha^2 (r_{(2)}) + \triangle (\alpha (s_{(1)})) \otimes \alpha (s_{(2)}) \cdot \alpha (a) \nonumber \\
&= \alpha (a) \cdot \alpha ( r_{(1)}) \otimes \alpha^{\otimes 2} (\triangle ( r_{(2)}) ) - \alpha^2 (s_{(1)}) \otimes s_{2} \bullet \triangle (a) - \alpha^2 (s_{(1)}) \otimes \triangle (s_{(2)}) \bullet a  \nonumber \\
&~~~ - a \bullet \triangle (r_{(1)}) \otimes \alpha^2 (r_{(2)}) - \triangle (a) \bullet r_{(1)} \otimes \alpha^2 (r_{(2)}) + \alpha^{\otimes 2} ( \triangle ( s_{(1)})) \otimes \alpha (s_{(2)}) \cdot \alpha (a) \nonumber \\
&= a \bullet (( \alpha \otimes \triangle)(r)) - s_{12} \triangle(a)_{23} - ((\alpha \otimes \triangle)(s)) \bullet a - a \bullet ((\triangle \otimes \alpha)(r)) - \triangle(a)_{12} r_{23} + ((\triangle \otimes \alpha )(s)) \bullet a  \nonumber \\
&= a \bullet (\alpha \otimes \triangle)^{-} (r) - (\alpha \otimes \triangle)^{-} (s) \bullet a - s_{12} \triangle(a)_{23} - \triangle(a)_{12} r_{23}.  \nonumber 
\end{align}
On the other hand, by writing $\triangle (a) = a_{(1)} \otimes a_{(2)}$, we get
\begin{align*}
&(\alpha \otimes \triangle')^{-} \circ \triangle (a) \\
&= (\alpha \otimes \triangle' - \triangle' \otimes \alpha ) (a_{(1)} \otimes a_{(2)}) \\
&= \alpha (a_{(1)}) \otimes a_{(2)} \cdot r_{(1)} \otimes \alpha (r_{(2)}) - \alpha (a_{(1)}) \otimes \alpha (s_{(1)}) \otimes s_{(2)} \cdot a_{(2)} \\
&~~~~ - a_{(1)} \cdot r_{(1)} \otimes \alpha (r_{(2)}) \otimes \alpha (a_{(2)}) + \alpha (s_{(1)}) \otimes s_{(2)} \cdot a_{(1)} \otimes \alpha (a_{(2)}) \\
&= \triangle(a)_{12} r_{23} - s_{23} \triangle(a)_{13} - \triangle (a)_{13} r_{12} + s_{12} \triangle(a)_{23}.
\end{align*}
By adding the above two identities, we get
\begin{align*}
&(\alpha \otimes \triangle)^{-} \triangle'(a) + (\alpha \otimes \triangle')^{-}\circ \triangle (a) \\
&= a \bullet ((\alpha \otimes \triangle)^{-} (r)) - ((\alpha \otimes \triangle)^{-} (s)) \bullet a - s_{23} \triangle(a)_{13} - \triangle(a)_{13} r_{12}.
\end{align*}
This is same as $a \bullet (r_{13}r_{12} - r_{12} r_{23} + s_{23} r_{13}) - (s_{13} r_{12} - s_{12} s_{23} + s_{23}s_{13}) \bullet a$ (equivalently, $(\triangle + \triangle')$ is Hom-coassociative) if and only if (\ref{perturb}) holds. Hence the proof. 
\end{proof}

The above theorem has some nice consequences. 

(i) (Taking $\alpha = \mathrm{id}$) Given a covariant bialgebra $(A, \mu, \triangle, \delta_1, \delta_2)$ and a pair $(r, s)$ of elements of $A^{\otimes 2}$, the theorem gives a necessary and sufficient condition for $(A, \mu, \triangle+ \triangle', \delta_1 + \delta_r, \delta_2 + \delta_s)$ to be a covariant bialgebra.

(ii) (Taking $\triangle = \delta_1 = \delta_2 = 0$) Given a Hom-associative algebra $(A, \alpha, \mu)$ and a pair $(r, s)$ of $\alpha$-invariant elements of $A^{\otimes 2}$, the tuple $(A, \alpha, \mu, \triangle', \delta_r, \delta_s)$ is a covariant Hom-bialgebra if and only if 
\begin{align*}
a \bullet ( r_{13}r_{12} - r_{12} r_{23} + s_{23} r_{13} ) = ( s_{13} r_{12} - s_{12} s_{23} + s_{23}s_{13}) \bullet a.
\end{align*}
Therefore, in this case, we recover Theorem \ref{quasi-hom-co}.

\vspace*{1cm}

\noindent {\bf Acknowledgements.} The research is supported by the fellowship of Indian Institute of Technology (IIT) Kanpur. The author thanks the Institute for support.

\end{document}